\documentclass[12pt]{amsart}
\usepackage{latexsym}
\usepackage{amssymb}
\usepackage[cp850]{inputenc}
\usepackage{epsfig}
\usepackage{amsthm}
\usepackage{amscd}
\usepackage{amsmath}
\usepackage{amsfonts}
\usepackage{graphics}
\usepackage[all]{xy}

  \newtheorem{theorem}{Theorem}[section]
  \newtheorem{proposition}[theorem]{Proposition}
  \newtheorem{corollary}[theorem]{Corollary}
  \newtheorem{lemma}[theorem]{Lemma}

  \theoremstyle{definition}

  \newtheorem*{claim*}{Claim}
  
  \newtheorem{example}[theorem]{Example}
  
  \newtheorem*{question*}{Question}
  \newtheorem*{answer*}{Answer}
  \newtheorem*{application*}{Application}

  \theoremstyle{remark}
  
  \newtheorem*{remark*}{Remark}
\newcommand{\QQ}{\mathbb{Q}}

\newcommand{\trunc}{\textup{trunc}}
\newcommand{\aut}{\textup{Aut}(F_r)}
\newcommand{\IA}{\textup{IA}_r}
\newcommand{\mE}{\mathcal{E}}
\newcommand{\ZZ}{\mathbb{Z}}

\newcommand{\comment}[1]{}

\newcommand{\rkg}{\textup{rank}(\mathfrak{g})}
\newcounter{problemnumber}
\begin{document}
%\thispagestyle{empty}
%\pagestyle{empty} 
%\noindent \bf Linear representations of the automorphisms of a free group on the homology of representation varieties. \rm
%
%
%\vspace {15mm}
%
%
%
%\vfill
%
%\clearpage
  \title    {Linear representations of $\aut$ on the homology of representation varieties.}
  \author   {Yael Algom-Kfir and Asaf Hadari}
  
  \date{\today}

  \begin{abstract} 
Let $G$ be a compact semisimple linear Lie group. We study the action of $\aut$ on the space $H_*(G^r; \QQ)$. We compute the image of this representation and prove that it only depends on the rank of $\mathfrak{g}$. We show that the kernel of this representation is always the Torrelli subgroup $\IA$ of $\aut$. 
  \end{abstract}
  \maketitle

\section{introduction}
Let $G$ be a group. The group $G^{r}$ can be naturally identified with the set of all homomorphisms $\rho: F_r \to G$, where $F_r$ is the free group of rank $r$. As such, the group $\Gamma = \textrm{Aut}(F_r)$ acts on $G^{r}$ by precomposition.  When $G$ is a topological group,  this action can be studied using topological tools. One example is to study the induced action of $\Gamma$ on $H_{*}(G^{r};\mathbb{Q})$. This gives rise to a linear representation of $\Gamma$, which we call $\mathcal{H}(G)$. In this paper we calculate the kernel and isomorphism class of this representation when $G$ is a compact semisimple Lie Group. \\

In \cite{CLM}, Capell, Lee, and Miller carry out a similar calculation, taking a surface group instead of $F_r$, a mapping class group instead of $\Gamma$, and restriciting to the case  on $G = SU(2)$. They find that the kernel of the action is an infinite index subgroup of the Torelli group.  In contrast to this, we find:

\begin{theorem} \label{kernelcalculation} Let $G$ be a compact semisimple Lie group. Then $\ker(\mathcal{H}(G)) = \textrm{IA}_r$, where $\textrm{IA}_r$ is the subgroup of $\Gamma$ that acts trivially on $F_r/[F_r,F_r] $.
\end{theorem}
\noindent One immediate consequence of this theorem, is that the representation $\mathcal{H}(G)$ can be thought of as a $\textrm{Out}(F_r)$ representation. Using the fact that $\Gamma / \textrm{IA}_r \cong \textrm{GL}_r(\mathbb{Z})$, we are able to get a precise description of these representations as $\textrm{GL}_r(\mathbb{Z})$-modules.  One main feature of this description is that the isomorphism class of  $\mathcal{H}(G)$ depends only on the rank of the Lie algebra of $G$. 

\begin{theorem}\label{imagecalculation} Let $G$ be a compact semisimple Lie group, and $\mathfrak{g} $ be its Lie algebra. Let $A = \frac{F_r}{[F_r,F_r]} \otimes \mathbb{Q}$, then, as a $\textrm{GL}_r(\mathbb{Z})$ module:
 
$$\mathcal{H}(G) \cong \bigotimes_{i = 1}^{\textup{rank}(\mathfrak{g})} \Lambda(A) $$
\end{theorem}

\noindent It is possible to generalize the construction of the representations $\mathcal{H}(G)$ in the following way: given a finite index subgroup $K < F_r$, we have a finite index subgroup $\Gamma_K < \Gamma_r$ such that $\Gamma_K (K) = K$, an inclusion $\Gamma_K \to \textrm{Aut}(K) $, and a representation:
$$\rho_K: \Gamma_K \to GL(H_{*}(G^{\textrm{rank}(K)};\mathbb{Q}))$$

\noindent We induce this representation to $\Gamma$ and define: 

$$\mathcal{H}_K(G) = \textrm{Ind}_{\Gamma_K}^{\Gamma}(\rho_K)$$

\noindent We prove a similar result, namely that:

\begin{theorem} \label{generalrepresentation}
Let G, G' be compact semisimple Lie groups, and let $K < F$ be a finite index subgroup. Then:

$$\ker(\mathcal{H}_K(G)) = \ker (\mathcal{H}_K(G')) $$

\noindent Furthermore, if $\textrm{rank}(\mathfrak{g}) = \textrm{rank}(\mathfrak{g}')$ then $$\mathcal{H}_K(G) \cong \mathcal{H}_K(G') $$

\end{theorem}

The image of these homomorphism were studied by Lubotzky and Grunewald in \cite{GL}, in the case where $G = SO(1)$. One of feature of this generalized construction is that the extended family of representations is asymptotically faithful, that is:  for every $\gamma \in \Gamma$ and for every $G$, $\exists K < F_r$ of finite index such that $\mathcal{H}_K(G)(\gamma) \neq Id$. This follows directly from the residual finiteness of $F_r$. Note that the generalized representations $\mathcal{H}_K(G)$ are no longer $\textrm{Out}(F_r)$ representations.

\section{The rational homology of Lie groups}
\subsection{Pontryagin and intersection structures}
\noindent
Given a compact  Lie group G, consider the continuous map $\mu: G \times G \to G$ given by $\mu(a,b) = ab$. This induces a linear map: $\mu_{*}: H(G; \mathbb{Q}) \otimes H(G; \mathbb{Q}) \to  H(G; \mathbb{Q})$. The map $\mu_{*}$ gives $H(G; \mathbb{Q})$ the structure of a ring, called the Pontryagin ring. We denote Pontryagin ring multiplication using the $\cdot$ symbol.

\noindent
The structure of $(H(G ; \mathbb{Q}), \cdot)$ is classic and has been calculated by Pontryagin in \cite{Struct}. For the sake of simplicity, we state it first for simple non-exceptional Lie groups. The root system of such Lie groups belong to the one of the infinite families: $A_{n}, B_{n}, C_{n}, D_{n}$, where $n$ is the rank of the Lie algebra of $G$ .  This root system entirely determines the above structure.  To each such root system, assign a finite sequence in the following way:

$$A_{n}: 3, 5, \ldots, 2n+1 $$
$$B_{n}:  3, 7, \dots, 4n-1$$
$$C_{n}:  3, 7, \ldots, 4n-1$$
$$D_{n}:  3, 7, \ldots , 4n-5, 2n - 1$$

\noindent
Note that in the case of $D_{n}$, the sequence is not arithmetic. Let $i_{1}, \ldots, i_{n}$ be the sequence assigned to $G$. Then:  

\begin{equation} \label{eq1}
H_*(G ; \mathbb{Q}) \cong H_*(S^{i_{1}} \times \ldots \times S^{i_{n}})
\end{equation}

\noindent where $S^{j}$ is the $j$-dimensional sphere. Given a basepoint \[ p_1 \times \ldots \times p_n \in  S^{i_{1}} \times \ldots \times S^{i_{n}}\] and an index $1 \leq l \leq n$, there is an  inclusion given by
\begin{equation}\label{eq_defn_I} 
\begin{array}{rl}
I_{l}:  S^{i_l} & \hookrightarrow S^{i_{1}} \times \ldots \times S^{i_{n}} \\
x & \to  (p_1, \ldots, p_{l-1}, x, p_{l+1}, \ldots, p_n)
\end{array}
\end{equation} 
The isomorphism in equation (\ref{eq1}) can be made explicit, by finding a $i_{l}$-dimensional submanifold of $G$ for every $l$ , which we call $T^l$ such that $s^i = [I_{l,*}(S^{i_l})]$ corresponds to $[T^l]$, where $[\cdot]$ denotes the image in homology.   

\noindent The Pontryagin structure is also known: it is a graded Grassman algebra with a unit and one generator of dimension $i_{k}$ for each $S^{i_k}$, which we denote $t^k$, where $t^k$ is the homology class of $T^k$. The set of generators will be denoted $\mathcal{P}(G)$. In the sequel, we will write $i_k = \textrm{dim}(t^k)$. Under the isomorphism in equation (\ref{eq1}), the Pontryagin structure corresponds to the direct product structure.

\begin{example} Let $G = SO_{3}(\mathbb{R})$. Then $H_{*}(G) \cong H_*(S^{3} \times S^{5})$. The Poincare polynomial of  $H_*(G)$ is $1 + x^{3} + x^{5} + x^{8}$. As generators for the Pontryagin structure, we have $t^{i} \in H_{i}(G)$ for $i = 1,2$, with dimensions $3,5$, and unit, which we call 1. The space $H_{8}(G)$ generated by $t^{1} \cdot t^{2} = -t^{2} \cdot t^{1}$.
\end{example} 

\noindent The situation for general semisimple groups $G$ (including the simple exceptional Lie groups) is similar. As in the simple case, $H_*(G)$ is isomorphic to the product of $\textup{rank}(\mathfrak{g})$ spheres, where $\mathfrak{g}$ is the Lie algebra of $G$. Furthermore, the Pontryagin structure is once again a Grassman algebra defined in the same way. 

\noindent In this article we are interested in the homology of two different groups: $G$ and $G^r$. We find it convenient to modify the notation above when considering the homology of $G^r$, despite the fact that it too is a semisimple Lie group. For $1 \leq i \leq r$ we have natural inclusions $\Delta_i: G \to G^{r}$ which induce maps $(\Delta_{i})_* : H_*(G) \to H_*(G^{r})$. Given $t^{j} \in H_*(G)$ we denote $(\Delta_{i})_*(t^{j}) = t^{j}_i$. The elements $t^{j}_{i}$ are the same generators of $H_*(G^r)$ discussed above, but are indexed differently.  A product of the $t^j_i$'s  is called a \emph{monomial}. Notice that the set of monomials is finite. 

\begin{example}Suppose $G = SU(2)^{n}$. The Poincare polynomial of $SU(2)$ is $1 + x^{3}$. The Poincare polynomial of $G$ is $\Sigma_{j = 1}^{n}  \left( \begin{array}{c}  n \\ j \end{array} \right) x^{3j}$. As a ring, it is generated by the elements $ t^{1}_{1} , \ldots, t^{1}_{n}$, all of which have dimension $3$.
\end{example}

\subsection{Diagonal inclusions} Give $1 \leq i  \leq k \leq r$,  the the diagonal inclusion $\Delta_{i,j}: G \to G^{r} $  maps $g \in G$ to the element with $g$ in the $i$ and $k$ components, and $Id$ in all other components. Notice that we allow the case $i =k$. \\
\medskip
\noindent
 The map assigning a group to its Pontryagin ring is functorial, and  the homomorphisms $\Delta_{i,k}$  induce ring homomorphisms $\Delta_{i,k}: (H(G), \cdot) \to H(G^{r}, \cdot)$. Thus, to give an explicit description of these maps, it is enough to calculate them for elements $t = t^j \in \mathcal{P}(G)$.   We first notice that if $i = k$, then $\Delta_{i,i}(t^j)  = t^j_i \in \mathcal{P}(G^{r})$. In what follows, assume $i \neq k$.

 \begin{proposition} \label{diaginc}
In the notation above, $\Delta_{i,k}(t^j) = t^j_i + t^j_k$ 
 \end{proposition}
\begin{proof} 
We  find it convenient to work in the setting of a product of spheres.  
Let $D = \textrm{dim}(t^j)$. The map $\Delta_{i,k}$ factors in the following way:

$$H_*(S^D) \xrightarrow{(\Delta_{1,2})_*} H_*(S^D \times S^D) \xrightarrow{I_j \times I_j} H_*(G \times G) \xrightarrow{\Delta_{i,i} \cdot \Delta_{k,k}} H_*(G \times \dots \times G)  $$

\noindent where the maps $I_i, I_k$ are defined in equation (\ref{eq_defn_I}). The vector space $H_D(S^D \times S^D)$ is two dimensional, and is generated by $a = \Delta_{1,1}([S^D])$ and $b = \Delta_{2,2}([S^D])$. The manifolds $\Delta_{1,1}(S^D)$ and $\Delta_{2,2}(S^D)$ each intersect $\Delta_{1,2}(S^D)$ once, so we have four possibilities  $(\Delta_{1,2})_*([S^D]) = \pm a \pm b$. Let $s: S^D \times S^D \to S^D \times S^D$ be given by $s(x,y) = (y,x)$. We have that $\Delta_{1,2} = s \circ \Delta_{1,2}$, so we are left with two options: $(\Delta_{1,2})_*([S^D]) = a + b$ or $-a - b$. Let $p: S^D \times S^D \to S^D$ be the projection to the first factor, given by $p(x,y) = x$. We have that $p \circ \Delta_{1,2} = Id$, which rules out $-a-b$. Thus $(\Delta_{1,2})_*([S^D]) = a + b$. Let $\phi = (\Delta_{i,i}\cdot \Delta_{k,k}) \circ (I_{j} \times I_j)$,  by definitions, it follows that  $\phi(a) = t^j_{i}$, $\phi(b) = t^j_k$, so $\phi(a+b) = t^j_i + t^j_k$, as claimed.
\end{proof}

\section{Calculating the action of generators of $\Gamma$}
Choose a basis $\{a_{1}, \ldots, a_{r}\}$ for $F_{r}$. A (right or left) Nielsen transformation with respect to this basis is an automorphism of the form:
$$N^{R}_{i,j}(a_{k}) = a_{k} a_{i}^{\delta_{j}(k)}$$
or 
$$N^{L}_{i,j}(a_{k}) = a_{i}^{\delta_{j}(k)} a_{k} $$
where $\delta_{i}(k)$ is the Kroenecker delta function.  The group $\textrm{Aut}(F_{r})$ has a index $2$ subgroup that is generated by Nielsen transformations. 

\noindent For $1 \leq k \leq r$ define $p_{k}: G^r \to G$ to be the $k$--th coordinate projection. Let $\trunc_k: G^r \to G^r$ by the rule: 
$$\trunc_k(g_1,\ldots,g_r) = (g_1, \dots, 1, \dots g_r)$$
where the 1 is in the $k$--th coordinate.

\begin{proposition} \label{Nielsenhom}
The maps $N^R_{i,j}$ and $N^L_{i,j}$ induce ring homomorphisms of the Pontryagin ring of $G$. Furthermore, these homomorphisms are both equal, and: 
$$N^R_{i,j}(t^k_l) = N^L_{i,j}(t^k_l) = \left\{ \begin{array}{l l} t^k_l & l \neq i \\[0.4cm]  t^k_i + t^k_j & l = i \end{array} \right. $$
 
\end{proposition}

\begin{proof} By definition $N^R_{i,j}$ and $N^L_{i,j}$ factor through the maps shown in the following commutative diagrams.

%\[ \xymatrix{ T \ar@/_/[ddr]_y \ar@/^/[drr]^x \ar@{.>}[dr]|-{(x,y)}\\ &X \times_Z Y \ar[d]^q \ar[r]_p & X\ar[d]_f \\ &Y \ar[r]^g &Z} \]

\[ \xymatrixcolsep{4pc} \xymatrix{ G^r \ar[r]^{\Delta_{1,2}} \ar@/_{1.5pc}/[rrrr]_{N^R_{i,j}} & G^r \times G^r  \ar[r]^{(\trunc_i,Id)} & G^r \times G^r  \ar[r]^{(Id, \Delta_{i,j} \circ p_i)} & G^r \times G^r  \ar[r]^{\mu} & G^r } \]

\[ \xymatrixcolsep{4pc} \xymatrix{ G^r \ar[r]^{\Delta_{1,2}} \ar@/_{1.5pc}/[rrrr]_{N^L_{i,j}} & G^r \times G^r  \ar[r]^{(Id,\trunc_i)} & G^r \times G^r  \ar[r]^{(\Delta_{i,j} \circ p_i, Id)} & G^r \times G^r  \ar[r]^{\mu} & G^r } \]

\noindent Where $\mu$ is the Pontryagin product. Now, let 
\[\Psi^R = (Id, \Delta_{i,j}\circ p_i)\circ (\trunc_i, Id)\circ \Delta_{1,2}\] and define $\Psi^L$ similarly. The maps $\Psi^R, \Psi^L$ are homomorphisms, and as such the induced maps on Homology are Pontryagin ring homomorphisms. Thus, to calculate the action of these homomorphisms, it suffices to calculate their action on generators. 

\noindent Notice that the inclusion of $G^r$ into $G^r \times G^r$ necessitates the addition of an index. We will use the notation $(t)_1$ or $(t)_2$ to denote the image of $t$ in $H_*(G^r \times G^r)$ under the maps $\Delta_{1,1}$, $\Delta_{2,2}$. By proposition \ref{diaginc}, we have that:

$$\Delta_{1,2}(t^k_l)  = (t^k_l)_1 + (t^k_l)_2$$

\noindent By definition, we have that $$(\trunc_i,Id)(t^k_l)_s = \delta_{l,i} \delta_{s,1} (t^k_l)_s$$
that is -  the map sends generators of the form $(t^k_i)_i$ to themselves, and all other generators to $0$. Once again, by proposition \ref{diaginc} we have that:

$$(Id, \Delta_{i,j} \circ p_i)(t^k_l)_s = \left \{ \begin{array} {l c} (t^k_l)_s & s= 1 \\[0.4cm] \delta_{l,i}\bigl[ (t^k_i)_2 + (t^k_j)_2 \bigr] & s = 2 \end{array} \right.  $$

\noindent Composing these maps, we get:
$$\Psi^R(t^k_l) = \left \{ \begin{array} {l c} (t^k_l)_1 & i \neq  l \\[0.4cm] (t^k_i)_2 + (t^k_j)_2  & i = l \end{array} \right. $$

\noindent Given $x,y \in H_*(G^r)$, the definition of Pontryagin multiplication gives that $$\mu\bigl((x)_1(y)_2 \bigl) = x\cdot y = \mu((x \cdot y)_1) $$ 

\noindent In other words, $\mu$ removes the outside index, and sends multiplication to multiplication. In particular, $\mu(t^k_l)_s = t^k_l$, so $N^R_{i,j}$ is as in the statement of the proposition. The same line of reasoning holds for $N^L_{i,j}$.
\end{proof}

Consider the permutation $\sigma = (1 2) \in S_r$. The permutation $\sigma$ acts on $F_r$ by exchanging the first and second generators. Nielsen proved \cite{NielsenPres} that the group $\Gamma$ is generated by $\sigma$ and by Nielsen transformations. The action of $\sigma$ on $H_*(G^r)$ is given by:

$$\sigma(t^k_l) =  t^k_{\sigma(l)}$$ 

\noindent This gives complete description of the action of a set of generators of $\Gamma$ on $H_*(G^r)$.

\begin{example} 
Let $G = SO_{3}(\mathbb{R})$, and $r = 3$. Generators of $H_*(G^3)$ are of the form $t^k_l$ where $1 \leq k \leq 2$, $1 \leq l \leq 3$. When $k = 1$, the dimension of the generator is 3, and when $k = 2$ the dimension is $5$. We have that:

$$ \begin{array}{ll}
N^R_{1,3}(t^1_1 t^1_3 t^2_1t^2_2) & = (t^1_1 + t^1_3)t^1_3(t^2_1 + t^2_3)t^2_2 \\[0.4 cm]
& = t^1_1 t^1_3 t^2_1t^2_2 + t^1_3 t^1_3 t^2_1 t^2_2 + t^1_1t^1_3t^2_3t^2_2 + t^1_3t^1_3t^2_3t^2_2 \\[0.4 cm]
& =  t^1_1 t^1_3 t^2_1t^2_2 + 0 - t^1_1t^1_3t^2_2t^2_3 + 0
\end{array}$$

\end{example}

\section{Describing the image and the kernel of $\mathcal{H}(\Gamma)$}

The group $\Gamma$ acts on the abelianization of $F_r$. Let $A = \frac{F_r}{[F_r,F_r] }\otimes \mathbb{Q}$ and let $\alpha: \Gamma \to GL_r(\mathbb{Q})$ be the homomorphism giving this action. \\

\noindent In this section we show that $\mathcal{H}$ factors through $\alpha$, and describe $\mathcal{H}$ as a $GL_r(\mathbb{Q})$-module.  Fix an index $k$, and consider the linear map $\mE_k: \Lambda(A) \to H_*(G^r)$ given on a basis by:

$$\mE_k (e_{j_1} \wedge \dots  \wedge e_{j_m}) = t^k_{j_1} \dots t^k_{j_m} $$
$$\mE_k(1_{\Lambda(A})) = 1_{H_*(G^r)} $$
\noindent Given a subset $S = \{i_1, \dots, i_l \} \subset \{1, \dots, \rkg \} $ , define the map $$\mathcal{E}_S: \bigotimes^{l} \Lambda(A) \to H_*(G^r)$$
by setting: 

$$\mathcal{E}_S(v_1 \otimes \dots \otimes v_l) = \prod_{j=1}^l \mathcal{E}_{i_j}(v_j)$$

\noindent The space $\bigotimes^l (\Lambda(A))$ is a $GL_r(\mathbb{Q})$-module and hence also a $\Gamma$-module. 

\begin{proposition} \label{equivmap} 
The maps $\mE_S$ defined above are $\Gamma$--equivariant. 
\end{proposition} 

\begin{proof} It is enough to show equivariance for a set of generators of $\Gamma$.  The action of a Nielsen transformation $N^R_{i,j}$ on $A$ is given in the standard basis by 

$$N^R_{i,j}(e_l) = \left \{ \begin{array}{ll} e_i + e_j & l = j \\ e_l & l \neq j  \end{array}  \right.$$

The action of $N^L_{i,j}$ is the same. The map $\sigma$ permutes the first and second coordinates. The action of $GL_r(\mathbb{Q})$ commutes with the wedge and tensor product, and hence $\mE_S \circ \alpha$ is exactly the same as the action calculated in Proposition \ref{Nielsenhom} .

\end{proof}

\begin{lemma} \label{commdiag}The following diagram commutes:
\[  \xymatrix{ \Gamma \ar[r] \ar[d] & GL(H_*(G^r)) \ar[d] \\
GL_r(\ZZ) \ar[r] & GL(\bigotimes_{i=1}^{\rkg} \Lambda(A) ) } \]
\end{lemma}

\begin{proof}
This follows from Proposition \ref{equivmap} by  letting $S = \{1, \dots, \rkg \} $. the map $\mE_S$ is an isomorphism, since both spaces have the same dimension. Indeed, they are both Grassman algebras on the same number of generators in the same dimensions. 

\end{proof}

\begin{corollary}  The kernel of the representation $\mathcal{H}_G$ is precisely $IA_r$.  The image depends only on $\rkg$. 
\end{corollary}
\begin{proof}
This follows from Lemma \ref{commdiag} and and the fact that $$\textup{ker}\left(GL_r(\mathbb{Z}) \to GL(\bigotimes_{i=1}^{\rkg} \Lambda(A) ) \right) = Id $$
because, for example $A \subseteq GL(\bigotimes_{i=1}^{\rkg} \Lambda(A) ) $
\end{proof}

\section{Describing the image and kernel of $\mathcal{H}^K(G)$}

Recall that $\Gamma_K = \{\gamma \in  \Gamma | \gamma K = K \}$. We can view the group $\Gamma_K$ as a subgroup of $\textup{Aut}(K)$. As such Lemma \ref{commdiag} gives the following corollary. 

\begin{corollary} The following diagram commutes:
\[  \xymatrix{ \Gamma_K \ar[r] \ar[d] & GL(H_*(G^{\textup{rank}(K)})) \ar[d] \\
GL_{\textup{rank}(K)}(\ZZ) \ar[r] & GL\left(\bigotimes_{i=1}^{\rkg} \Lambda \bigl( \frac{K}{[K,K]} \otimes \mathbb{Q} \bigr) \right) } \]
\end{corollary}

\noindent Recall that $\mathcal{H}_K(G) = \textrm{Ind}_{\Gamma_K}^{\Gamma}(\rho_K)$

\begin{corollary} \label{equivres} We have that:
$$\ker (\mathcal{H}_K(G)) = \Gamma_K \cap \textup{IA}(K) $$

\noindent Where $IA(K)$ is the subgroup of $\textup{Aut}(K)$ that acts trivially on the abelianization of $K$. \\
Furthermore, if $G'$ is another compact semisimple Lie group whose Lie algebra has the same rank as $\mathfrak{g}$ then $\mathcal{H}_K(G') \cong \mathcal{H}_K(G) $

\end{corollary}
\begin{proof}
For the first claim, we have that since $\mathcal{H}_K(G)$ is a representation on $\Gamma$ induced from $K$, it is clear that $\ker(\mathcal{H}_K(G)) \subset \Gamma_K$. By corollary \ref{equivres}, we have that $\ker(\mathcal{H}_K(G)) \cap \Gamma_K =  \textup{IA}(K) \cap \Gamma_K$.

\noindent For the second claim, it follows directly from the diagram, since we are inducing isomorphic representations from the same groups. 
\end{proof}

\bibliographystyle{alpha}
\bibliography{ref}

\end{document}